\documentclass[11pt]{article}
\usepackage{graphicx}
\usepackage{amsthm, amsmath, amssymb, tikz, bm}
\usepackage{mathtools}
\usepackage{xifthen}
\usepackage{caption}
\usepackage[colorlinks=true,
linkcolor=blue,citecolor=blue,
urlcolor=blue]{hyperref}
\usepackage[margin=0.8in]{geometry} 

\usepackage[shortlabels]{enumitem}
\usepackage{todonotes}
\usetikzlibrary{calc,shapes, backgrounds}
\allowdisplaybreaks

\newtheorem{lemma}{Lemma}
\newtheorem{theorem}{Theorem}

\newcommand{\cE}{\mathcal{E}}
\newcommand{\cH}{\mathcal{H}}

\newcommand{\thmref}[1]{Theorem~\ref{thm:#1}}

\title{Spectral radius and Hamiltonicity of uniform hypergraphs}
\author{George Brooks
\thanks{University of South Carolina, Columbia, SC. ({\tt ghbrooks@email.sc.edu}). Partially supported by NSF DMS 2038080 grant.}
\and
William Linz \thanks{University of South Carolina, Columbia, SC. ({\tt wlinz@mailbox.sc.edu}). Partially supported by NSF DMS 2038080 grant.}
\and
Ruth Luo \thanks{University of South Carolina, Columbia, SC. ({\tt ruthluo@sc.edu}).}
}

\begin{document}

\maketitle

\abstract{Let $n$ and $r$ be integers with $n-2\ge r\ge 3$. We prove that any $r$-uniform hypergraph $\mathcal{H}$ on $n$ vertices with spectral radius $\lambda(\mathcal{H}) > \binom{n-2}{r-1}$ must contain a Hamiltonian Berge cycle unless $\mathcal{H}$ is the  complete graph $K_{n-1}^r$ with one additional edge. This generalizes a result proved by Fiedler and Nikiforov for graphs. As part of our proof, we show that if $|\mathcal{H}| > \binom{n-1}{r}$, then $\mathcal{H}$ contains a Hamiltonian Berge cycle unless $\mathcal{H}$ is the complete graph $K_{n-1}^r$ with one additional edge, generalizing a classical theorem for graphs. 
}

\section{Introduction}

There is a long history of finding sufficient conditions which guarantee Hamiltonicity in graphs. A classic result in this area is Dirac's theorem~\cite{Dir}, which states that any graph on $n$ vertices with minimum degree at least $n/2$ must contain a Hamiltonian cycle.

Fiedler and Nikiforov~\cite{FN} gave a sufficient condition for Hamiltonicity based on a graph's spectral radius. Recall that the \emph{spectral radius} of a graph $G$ is defined as the largest eigenvalue of its adjacency matrix. 

\begin{theorem}[Fiedler-Nikiforov]\label{thm:fn}
Let $G$ be a graph of order $n$ and spectral radius $\lambda(G)$. If $\lambda(G) \ge n-2$, then $G$ contain a Hamiltonian path unless $G\cong K_{n-1} + v$. If $\lambda(G) > n-2$, then $G$ contains a Hamiltonian cycle unless $G\cong K_{n-1} + e$.
\end{theorem}

A number of other authors~\cite{LLT, Nik, NG} have extended \thmref{fn} or given further spectral conditions which imply Hamiltonicity. 

In this note, we extend \thmref{fn} to the hypergraph setting. An \emph{$r$-graph} or $r$-uniform hypergraph $\cH$ on $n$ vertices is a pair $(V, \cE)$, where $|V| = n$ and $\cE\subseteq \binom{V}{r}$, where $\binom{V}{r}$ is the set of $r$-element subsets of $V$.  The cardinality of the edge-set $\cE$ is denoted by $|\cH|$. For a vertex $x\in V$, the \emph{degree} $d(x) := |\{E: x\in E\in \cE\}|$ is the number of edges in $\cE$ which contain $x$. The \emph{minimum degree} is defined to be $\delta(\cH):= \min_{x\in V}d(x)$. The \emph{complete $r$-graph} on $n$ vertices $K_n^r$ is the $r$-graph whose edge-set is the set of all $r$-element subsets of a vertex set of size $n$. 

Following Cooper and Dutle~\cite{CD}, the \emph{spectral radius} $\lambda(\cH)$ of an $r$-graph $\cH$ is defined as follows. The polynomial form $P_{\cH}({\bf x}): \mathbb{R}^n\rightarrow \mathbb{R}$ associated with the $r$-graph $\cH$ is defined for any vector ${\bf x} = (x_1, \ldots, x_n) \in \mathbb{R}^n$ by 
\[P_{\cH}({\bf x}) = r \sum_{\{i_1, i_2, \ldots, i_r\}\in E(\cH)}x_{i_1}\ldots x_{i_r}.\]

The \emph{spectral radius} $\lambda(\cH)$ is defined to be the maximum of $P_{\cH}({\bf x})$ over all unit vectors in the $\ell_r$-norm, i.e.
\[\lambda(\cH) = \max_{\lVert{\bf x}\rVert_r = 1}P_{\cH}({\bf x}).\]

(In other works \cite{NikAnalytic} the polynomial form is defined by $P_{\cH}({\bf x}) = r! \sum_{\{i_1, i_2, \ldots, i_r\}\in E(\cH)}x_{i_1}\ldots x_{i_r}$, but this of course only changes the spectral radius by a constant factor). 

A \emph{Berge cycle} of length $\ell$ in an $r$-graph is a list of $\ell$ distinct edges $e_1, \ldots, e_{\ell}$ and $\ell$ distinct vertices $v_1, \ldots v_{\ell}$ such that $v_i\in e_i\cap e_{i+1}$ for $1\le i\le \ell-1$ and $v_{\ell} \in e_{\ell} \cap e_1$. A \emph{Berge path} of length $\ell$ is obtained from a Berge cycle of length $\ell$ by deleting one of the edges. A \emph{Hamiltonian Berge cycle} and a \emph{Hamiltonian Berge path} are a Berge cycle and Berge path, respectively, which span all vertices of the $r$-graph; an $r$-graph containing a Hamiltonian Berge cycle is \emph{Hamiltonian}. An $r$-graph is \emph{Hamiltonian-connected} if it contains a Hamiltonian Berge path between any pair of vertices. 
  
We now state our $r$-graph version of \thmref{fn}. 
\begin{theorem}\label{thm:specrad}
Let $r\ge 3$ be an integer and suppose $n  \ge r+2$. Let $\cH$ be an $r$-graph on $n$ vertices. If $\lambda(\cH) \ge \binom{n-2}{r-1}$, then $\cH$ contains a Hamiltonian Berge path unless $\cH\cong K_{n-1}^r+v$. If $\lambda(\cH) > \binom{n-2}{r-1}$, then $\cH$ contains a Hamiltonian Berge cycle unless $\cH \cong K_{n-1}^r + e$. 
\end{theorem}

Here, $K_{n-1}^r+v$ denotes the union of an $r$-uniform, $(n-1)$-clique and an isolated vertex $v$, and $K_{n-1}^r+e$ is a $K_{n-1}^r+v$ with an additional edge incident to $v$.

To establish \thmref{specrad}, we first prove a sufficient condition for the existence of a Hamiltonian Berge path or cycle based on the number of edges $|\cH|$ in the hypergraph $\cH$.

\begin{theorem}\label{thm:edges}
Let $r\ge 3$ be an integer and suppose $n \ge r+2$. Let $\cH$ be an $r$-graph on $n$ vertices. If $|\cH| \ge \binom{n-1}{r}$, then $\cH$ contains a Hamiltonian Berge path unless $\cH\cong K_{n-1}^r+v$. If $|\cH| > \binom{n-1}{r}$, then $\cH$ contains a Hamiltonian Berge cycle unless $\cH \cong K_{n-1}^r + e$. 
\end{theorem}

\thmref{edges} extends a classical result that any graph $G$ with more than $\binom{n-1}{2} + 1$ edges must contain a Hamiltonian cycle. To our knowledge, \thmref{edges} has not appeared explicitly in the literature. We demonstrate that it follows from Dirac-type threshold results for $r$-graphs previously proven by Kostochka, Luo and McCourt~\cite{KLM1, KLM2}. \thmref{edges} is the key combinatorial input needed to prove \thmref{specrad}, which is deduced using an upper bound on the spectral radius of an $r$-graph proven by Bai and Lu~\cite{BL}. 

We note the condition $n\ge r+2$ in \thmref{specrad} and \thmref{edges} is best possible for $r\ge 3$. The only $r$-graph with $r+1$ vertices which contains a Hamiltonian Berge cycle is the complete graph $K_{r+1}^r$, since it is the only such $r$-graph which even has $r+1$ edges. Therefore, any noncomplete $r$-graph on $r+1$ vertices with at least three edges satisfies the conditions of \thmref{specrad} and \thmref{edges}, but does not contain a Hamiltonian Berge cycle. 

It would be interesting to find further spectral conditions which imply Hamiltonicity in $r$-graphs. Ning and Ge~\cite{NG} extended \thmref{fn} by proving a spectral sufficient condition for graphs $G$ with $\delta(G) \ge 2$ to be Hamiltonian; is there a similar extension of \thmref{specrad} for $r$-graphs? One could also attempt to find a spectral sufficient condition which guarantees pancyclicity for Berge cycles in $r$-graphs. Spectral sufficient conditions for pancyclicity in graphs have been previously studied~\cite{XYAAYC}. Moreover, Dirac-type minimum degree conditions for Berge pancyclicity have been studied in~\cite{bpan1,bpan2}.

\section{Proof of \thmref{edges}}

In this section, we prove \thmref{edges}. We first handle the special case $n=r+2$. 

\begin{lemma}\label{lem:r+2lem}
Let $n \geq 5$. If $\cH$ is an $n$-vertex, $(n-2)$-graph with $|\cH| \geq n$, then $\cH$ is Hamiltonian unless $\cH \cong K_{n-1}^{n-2} + e$. 
\end{lemma}

\begin{proof}
For $5\le n\le 8$, we check by computer that every $n$-vertex, $(n-2)$-graph with $n$ edges is Hamiltonian except for $K_{n-1}^{n-2} + e$ and every $n$-vertex, $(n-2)$-graph with $n+1$ edges is Hamiltonian. We do this by modeling\footnote{The code is available from the authors upon request.}  the problem of finding a Hamiltonian Berge cycle in a $r$-graph as a \emph{constraint satisfaction problem} (CSP), which is done by a straightforward adaptation of a CSP used for finding a Hamiltonian cycle in a graph. We use the CSP solver OR-Tools CP-SAT \cite{cpsat} to solve the CSP for each $n$-vertex, $(n-2)$-graph with $n$ or $n+1$ edges and $5\le n\le 8$. So assume that $n\ge 9$. 

If there exists $v \in \cH$ with $d(v) = 1$, then $\cH-v$ has at most ${n-1 \choose n-2}$ edges. It follows immediately that $|\cH| \leq {n-1 \choose n-2} + 1$ with equality only if $\cH \cong K_{n-1}^{n-2} + e$. Now suppose $\delta(\cH) \geq 2$.
If there exists a pair of vertices $u$ and $v$ that together span at most $n-2$ edges, then $\cH - u- v$ contains at most one edge, and hence $|\cH| < n$, a contradiction.

Consider an edge-maximal counterexample $\cH$. Then $\cH$ has at least $n$ edges and is not Hamiltonian, but the addition of any new edge creates a Hamiltonian Berge cycle. In particular, $\cH$ must have a Hamiltonian Berge path. Since each such path $P$ uses $n-1 < |\cH|$ edges, there exists an edge $e^{(P)}$ that is not used in $P$. 

We first show that among all such pairs $(P, e^{(P)})$, say $P = v_1, e_1, \ldots, e_{n-1}, v_n$, we may choose one with $v_1 \in e^{(P)}$. Suppose that $v_1 \not \in e^{(P)}$. By symmetry, every edge $e \notin E(P)$ contains neither $v_1$ nor $v_n$. Thus $e^{(P)}$ is the only edge outside of $P$ and $e^{(P)} = \{v_2, \ldots, v_{n-1}\}$. Without loss of generality, let $d(v_1) \geq d(v_n)$. Since every pair of vertices span at least $n-1$ edges, we have $d(v_1) + d(v_n) \geq n-1$ and so $d(v_1) \geq 3$. There exists some $e_i$, $i \notin \{1,n-1\}$ containing $v_1$. Then the pair $(P', e_i)$ with 
  \[P' = v_1, e_1, \ldots, v_{i}, e^{(P)}, v_{i+1}, e_{i+1}, \ldots, v_n,\] is chosen so that $v_1 \in e^{(P')}$.
  
Let $(P,e^{(P)})$ be such that $P$ is a Hamiltonian Berge path with $v_1 \in e^{(P)}$. If $v_n \in e^{(P)}$ as well then $\cH$ has a Hamiltonian Berge cycle. Suppose $v_n \not \in e^{(P)}$.

	Let $J = \{j: v_n \in e_j\}$. If $J = \{ n-1\}$, then there exists an edge $f^{(P)}$ with $v_n \in f^{(P)}$  which is not used in $P$ since $d(v_n)\geq 2$. Clearly $|e^{(P)} \cap f^{(P)}|\geq n-4$ and  $e^{(P)}\neq f^{(P)}$ since $v_n \notin e^{(P)}$ . Since $\cH$ is $(n-2)$-uniform with $n\geq 9$, there exists a pair of vertices $\{v_i,v_{i+1}\} \subset e^{(P)} \cap f^{(P)}$ and $\cH$ has a Hamiltonian Berge cycle
\[
v_1, \ldots, v_i, f^{(P)}, v_n, e_{n-1}, v_{n-1}, \ldots, v_{i+1}, e^{(P)}, v_1.
\]
  
Hence, we may assume $|J| > 1 $. If there exists $j \in J$ such that $v_{j+1} \in e^{(P)}$ then $\cH$ has a Hamiltonian Berge cycle
 \[v_1, \ldots, v_j, e_j, v_n, e_{n-1}, v_{n-1}, \ldots, v_{j+1}, e^{(P)}, v_1.\]It follows that
 \begin{equation}\label{disj}
\{v_{j+1}: j \in J\} \cap e^{(P)} = \emptyset. 
 \end{equation}

Let $e_i$ be an edge containing $v_n$ with $i \neq n-1$.  
Define $J' = \{j: \{v_j, v_{j+1}\} \subset e^{(P)}\}$. Since $v_n \notin e^{(P)}$, $|J'| \geq |e^{(P)}|-1 - 1 = n-4$. By~\eqref{disj}, for all $j \in J'$, $v_n \notin e_j$, as otherwise $j\in J$ and $v_{j+1} \in e^{(P)}$. Moreover, at most $2$ edges $e_j$ with $j \in J'$ do not contain $v_i$ or $v_{i+1}$ since $\cH$ is $(n-2)$-uniform. Thus, there exists at least $|J'| - 2 \geq n-6\geq 3$ edges $e_j$ with $j \in J'$ and $\{v_i,v_{i+1}\} \in e_j$. 
Let $J'' \subseteq J$ be all such $j$. 
 
Since $|e_i| = n-2$ and $|J''| \geq 3$, there exists $k \in J''$ such that $v_{k} \in e_i$. If $i < k$, then $\cH$ has a Hamiltonian Berge cycle
\[v_1, \ldots, v_i, e_k, v_{i+1}, e_{i+1}, \ldots, v_k, e_i, v_n, e_{n-1}, \ldots, v_{k+1}, e^{(P)}, v_1.\]
If $i > k$, then we instead get a Hamiltonian Berge cycle
\[v_1, \ldots, v_k, e_i, v_n, e_{n-1}, \ldots, v_{i+1}, e_k, v_i, \ldots, v_{k+1}, e^{(P)}, v_1.\]
\end{proof}

We will use the following Dirac-type results of Kostochka, Luo and McCourt~\cite[Theorem 1.3 and 1.4]{KLM2}. 

\begin{theorem}[Kostochka-Luo-McCourt]\label{thm:klm}
Let $n > r\ge 3$. Suppose $\cH$ is an $n$-vertex, $r$-graph. 
\begin{enumerate}
\item[(1)] If (i) $r\le (n-1)/2$ and $\delta(\cH)\ge \binom{\lfloor{(n-1)/2\rfloor}}{r-1} + 1$ or (ii) $n-1 \geq r\ge n/2$ and $\delta(\cH) \ge r$, then $\cH$ contains a Hamiltonian Berge cycle. 
\item[(2)] If (i) $r\le n/2$ and $\delta (\cH) \ge \binom{\lfloor{n/2\rfloor}}{r-1} + 1$ or (ii) $n-1\ge r > n/2 \ge 3$ and $\delta(\cH) \ge r-1$, then $\cH$ is Hamiltonian-connected. 
\end{enumerate}
\end{theorem}

We use \thmref{klm} to prove \thmref{edges} for $n\ge r+3$.  

\begin{proof}[Proof of \thmref{edges}]
 We first assume $n\ge 2r+1$. Let $\cH$ be an $n$-vertex $r$-graph and assume that $\cH$ does not contain a Hamiltonian Berge cycle. Then, by \thmref{klm}(1)(i), there is a vertex $x$ with degree at most $\binom{\lfloor{(n-1)/2\rfloor}}{r-1}$. If $d(x) = 1$, then it follows immediately that $|\cH|\le \binom{n-1}{r} + 1$ and $\cH\cong K_{n-1}^r + e$ if $|\cH| = \binom{n-1}{r} + 1$. Otherwise, if $d(x) \ge 2$, then there exist distinct vertices $u, v \in V - \{x\}$ such that there are distinct edges containing the $2$-sets $\{x, u\}$ and $\{x, v\}$ respectively. If $\cH-x$ is Hamiltonian-connected, then there is a Hamiltonian Berge path from $u$ to $v$, and we can connect $x$ to this path to obtain a Hamiltonian Berge cycle. So $\cH-x$ cannot be Hamiltonian-connected. Hence, by \thmref{klm}(2)(i), there is a vertex $y$ in $\cH-x$ with degree at most $\binom{\lfloor{(n-1)/2\rfloor}}{r-1}$ in $\cH - x$. Therefore, using the convexity of binomial coefficients, we have
\begin{align*}
	|\cH|&= |\cH-x-y| + d_{\cH}(x) + d_{\cH-x}(y)\\
	&\leq \binom{n-2}{r} + \binom{\lfloor{(n-1)/2\rfloor}}{r-1} + \binom{\lfloor{(n-1)/2\rfloor}}{r-1}\\
	&\leq \binom{n-2}{r} + \binom{\lfloor{(n-1)/2\rfloor}+\lfloor{(n-2)/2\rfloor}}{r-1}+\binom{\lfloor{(n-1)/2\rfloor}-\lfloor{(n-2)/2\rfloor}}{r-1}\\
	&= \binom{n-2}{r} + \binom{n-2}{r-1}\\
	&= \binom{n-1}{r}.
\end{align*}
If $r+3\leq n\leq2r$, then by using the same argument with the minimum degree bounds given by \thmref{klm}(1)(ii) and \thmref{klm}(2)(ii), we similarly obtain 
\begin{align*}
	|\cH|&= |\cH-x-y| + d_{\cH}(x) + d_{\cH-x}(y)\\
	&\leq \binom{n-2}{r} + (r-1) + (r-2) \leq \binom{n-1}{r}.
\end{align*}
The remaining case $n=r+2$ follows from Lemma~\ref{lem:r+2lem}. 

We have proved that if $|\cH| \ge \binom{n-1}{r} + 1$, then $\cH$ contains a Hamiltonian Berge cycle unless $\cH \cong K_{n-1}^r + e$. The hypergraph $K_{n-1}^r + e$ has a Hamiltonian Berge path with the degree one vertex as one of its endpoints. Assume $|\cH| = \binom{n-1}{r}$. If $\cH \cong K_{n-1}^r + v$, then $\cH$ does not have a Hamiltonian Berge path as it has an isolated vertex. Otherwise, there exists some $e' \notin \cE(\cH)$ such that $\cH + e' \not\cong K_{n-1}^r + e$, so that in particular $\cH+e'$ contains a Hamiltonian Berge cycle. It follows that $\cH$ contains a Hamiltonian Berge path. 
\end{proof}

\section{Proof of \thmref{specrad}}

The following result of Bai and Lu~\cite{BL} bounds the spectral radius of a hypergraph in terms of the number of edges. For a fixed integer $r$, define the polynomial $p_r(x) = \binom{x}{r} = \frac{x(x-1)\cdots (x-r+1)}{r!}$. The polynomial $p_r$ is an increasing function on the interval $[r-1, \infty)$; furthermore, its inverse function $p_r^{-1}:[0, \infty) \rightarrow [r-1, \infty)$ exists and is also increasing. Note in particular that $p_r^{-1}(\binom{a}{r}) = a$ for every integer $a\ge r$. Following Bai and Lu, define the function  \[f_r(x):= p_{r-1}(p_r^{-1}(x) - 1).\]

\begin{theorem}[Bai-Lu]\label{thm:bailu}
For fixed $r\ge 2$, let $\cH$ be an $r$-graph with $m$ edges. Then, 
\[\lambda(\cH) \le f_r(m),\]
and equality holds if and only if $m = \binom{k}{r}$ and $\cH$ is the union of the complete graph $K_{k}^r$ and possibly some isolated vertices. 
\end{theorem}

This is a generalization of the classic bound of Stanley~\cite{Stan}, which states that for a graph with $m$ edges, the spectral radius is bounded above by $\frac{\sqrt{1+8m} -1}{2}$. 

We proceed with the proof of \thmref{specrad}. 

\begin{proof}[Proof of \thmref{specrad}]
	Let $\cH$ be an $n$-vertex $r$-graph with $\lambda(\cH) \ge \binom{n-2}{r-1}$ and $m$ edges. From Bai and Lu's inequality in \thmref{bailu}, we have
	\begin{alignat*}{4}
			&& \binom{n-2}{r-1} \leq \lambda(\cH)  &\leq  p_{r-1}(p_r^{-1}(m) - 1), \\
		\Rightarrow\quad&& p^{-1}_{r-1} \left(\binom{n-2}{r-1} \right)  &\leq p^{-1}_r (m) - 1,\\
		\Rightarrow\quad&& n-2 +1 &\leq   p^{-1}_r (m),\\
		\Rightarrow\quad&& {n-1 \choose r} &\leq  m.
	\end{alignat*}
	Hence, by \thmref{edges}, $\cH$ contains a Hamiltonian Berge path unless $\cH \cong K_{n-1}^r + v$. If the inequality is strict in our assumption, it follows from \thmref{edges} that $\cH$ contains a Hamiltonian Berge cycle unless $\cH \cong K_{n-1}^{r} + e$. 
\end{proof}

\end{document}